\documentclass[]{amsart}

\usepackage{hyperref}
\hypersetup{%
	pdfpagemode={UseOutlines},
	bookmarksopen,
	pdfstartview={FitH},
	colorlinks,
	linkcolor={blue},
	citecolor={blue},
	urlcolor={blue}
}

\usepackage{geometry} 

\usepackage{dcolumn}
\newcolumntype{d}{D{.}{.}{-1} } 

\usepackage{amsmath, amssymb, amsthm}
\usepackage{mathrsfs}

\DeclareMathAlphabet{\mathdutchcal}{U}{dutchcal}{m}{n}
\SetMathAlphabet{\mathdutchcal}{bold}{U}{dutchcal}{b}{n}
\DeclareMathAlphabet{\mathdutchbcal}{U}{dutchcal}{b}{n}

\usepackage{esint}
\usepackage{enumitem}
\usepackage{tikz}
\usepackage{hyperref}
\usepackage{systeme}
\usepackage{mathtools}
\hypersetup{colorlinks, linkcolor=blue, citecolor=blue}
\usepackage{cancel}
\usepackage{aligned-overset}
\usepackage{color}
\usepackage[
backend=bibtex,
style=numeric,
sorting=nyt
]{biblatex}
\numberwithin{equation}{section}

\newtheorem{teo}{Theorem}[section]

\newtheorem{cor}[teo]{Corollary}
\newtheorem{prop}[teo]{Proposition}
\theoremstyle{definition}

\theoremstyle{remark}

\newcommand{\R}{\mathbb{R}} 
\newcommand{\N}{\mathbb{N}} 
\newcommand{\V}{\mathcal{V}} 
\newcommand{\Fock}{\mathcal{F}} 
\newcommand{\C}{\mathbb{C}} 
\newcommand{\Barg}{\mathcal{B}} 

\newcommand{\Log}{\ensuremath{\mathrm{Log}_-}}

\newcommand{\pfrac}[2]{\frac{\partial #1}{\partial #2}}

\title{A new optimal estimate for the norm of time-frequency localization operators}
\author{FEDERICO RICCARDI}

\begin{document}
	
	\begin{abstract}
		\noindent In this paper we provide an optimal estimate for the operator norm of time-frequency localization operators with Gaussian window $L_{F,\varphi} : L^2(\R^d) \rightarrow L^2(\R^d)$, under the assumption that $F \in L^p(\R^{2d}) \cap L^q(\R^{2d})$ for some $p$ and $q$ in $(1,+\infty)$. We are also able to characterize optimal weight functions, whose shape turns out to depend on the ratio $\|F\|_q / \|F\|_p$. Roughly speaking, if this ratio is ``sufficiently large'' or ``sufficiently small'' optimal weight functions are certain Gaussians, while if it is in the intermediate regime the optimal functions are no longer Gaussian functions. As an application we extend Lieb's uncertainty inequality to the space $L^p + L^q$.
	\end{abstract}

\maketitle

\section{Introduction}

	\emph{Time-frequency localization operators} were originally introduced by Daubechies \cite{daubechies} and by Berezin \cite{berezin} and, roughly speaking, are designed in order to localize a function simultaneously in time and frequency. Defining the (normalized in $L^2$) ``Gaussian window'' as 
	\begin{equation}\label{eq: Gaussian window}
		\varphi(t) = 2^{d/4}e^{-\pi |t|^2}, \quad t \in \R^d,
	\end{equation}
	the main tool to introduce these operators is the short-time Fourier transform with Gaussian window, which is defined as
	\begin{equation}\label{eq: definition STFT}
		\V_{\varphi} f(x, \omega) = \int_{\R^d} f(t) \varphi(t-x) e^{- 2 \pi i \omega \cdot t} \, dt, \quad x,\omega \in \R^d,
	\end{equation}
	for every $f \in L^2(\R^d)$. Then, given a function $F : \R^{2d} \rightarrow \C$, we can define the time-frequency localization operator with window $\varphi$ and weight function $F$ as
	\begin{equation}\label{eq: definition localization operator}
		L_{F,\varphi} \coloneqq \V_{\varphi}^* F \V_{\varphi}.
	\end{equation}
	Since their introduction, these operators were intensively studied, in particular regarding properties such as boundedness, compactness, Schatten properties and asymptotic for the eigenvalues (see, for example, \cite{cordero, fernandez, wong} for some classical result or \cite{abreu, abreu2021donoho, dias2022uncertainty, galbis2022norm, luef, nicola2023uncertainty, nicolatilli_fk, nicolatilli_norm} for more recent results).
	However, results concerning optimal estimates and optimal weight functions are still few. In order to present the state of the art, for the introduction we confine ourself in the 1-dimensional case. Moreover, so as to be more in line with our approach, we will state the results under the form of constrained optimization problems. Therefore, we will deal with the problem of finding the optimal constant $C$ such that
	\begin{equation*}
		\|L_{F,\varphi}\| \leq C
	\end{equation*}
	and those weight functions $F$ that achieve equality among all possible weight functions satisfying one or more constraints for some $L^p$ norm.
	
	The following result can be obtained, through a duality argument, using Lieb's inequality for the STFT \cite{lieb_entropy}, which in dimension 1 is
	\begin{equation}\label{eq:Lieb's inequality}
		\| \V_{\varphi} f \|_p^p \leq \dfrac{2}{p} \, \|f\|_2^p,
	\end{equation}
	for $p \in [2, +\infty)$, while optimal weight functions can be obtained thanks to the fact due to Carlen \cite{carlen} that functions achieving equaility in \eqref{eq:Lieb's inequality} are of the kind \eqref{eq: optimal function Nicola-Tilli Faber-Krahn for STFT}.
	
	In the following we let:
	\begin{equation*}
		\kappa_p \coloneqq \dfrac{p-1}{p}.
	\end{equation*}	
	\begin{teo}[Lieb's uncertainty inequality - dual form]\label{th:Lieb's inequality}
		Let $p \in (1,+\infty)$ and $A \in (0,+\infty)$. If $F \in L^p(\R^{2})$ with $\|F\|_p \leq ~ A$, then:
		\begin{equation}\label{eq: estimate Lieb}
			\|L_{F,\varphi}\| \leq \kappa_p^{\kappa_p} A
		\end{equation}
		and equality is achieved if and only if, for some $\theta \in \R$ and some $z_0 \in \R^2$,
		\begin{equation}\label{eq: optimal function Lieb}
			F(z) = A \kappa_p^{-\frac{1}{p}} e^{i \theta} e^{-\frac{\pi}{p-1}|z-z_0|^2}, \quad z \in \R^2.
		\end{equation}
	\end{teo}
	According to this theorem, optimal weight functions (up to translations and a phase factor) are Gaussians.
	
	A further result was obtained only recently by Nicola and Tilli in \cite{nicolatilli_norm}, where they considered the case $F \in L^p(\R^{2}) \cap L^{\infty}(\R^{2})$ for $p \in [1,+\infty)$. In this paper we are interested only in the case $p>1$ and we refer the reader to the paper for the full theorem.
	\begin{teo}[\cite{nicolatilli_norm}]\label{th:Nicola-Tilli norm}
		Let $p \in (1,+\infty)$ and $A,B \in (0,+\infty)$. Let $F \in L^p(\R^{2}) \cap L^{\infty} (\R^{2})$ with $\|F\|_{\infty} \leq A$ and $\|F\|_p \leq B$.
		\begin{enumerate}[label*=(\roman{*})]
			\item If $B/A \leq \kappa_p^{1/p}$ then
			\begin{equation}\label{eq: estimate Nicola-Tilli first case}
				\|L_{F,\varphi}\| \leq \kappa_p^{\kappa_p} B,
			\end{equation}
			with equality if and only if, for some $\theta \in \R$ and $z_0 \in \R^2$,
			\begin{equation}\label{eq: optimal function Nicola-Tilli first case}
				F(z) = B \kappa_p^{-\frac{1}{p}} e^{i \theta} e^{-\frac{\pi}{p-1}|z-z_0|^2}, \quad z \in \R^2.
			\end{equation}
			\item If $B/A > \kappa_p^{1/p}$ then
			\begin{equation}\label{eq: estimate Nicola-Tilli second case}
				\|L_{F,\varphi}\| \leq A \left(1 - \dfrac{e^{\kappa_p - (B / A)^p}}{p}\right),
			\end{equation}
			with equality if and only if, for some $\theta \in \R$ and $z_0 \in \R^2$,
			\begin{equation}\label{eq: optimal function Nicola-Tilli second case}
				F(z) = e^{i \theta} \min \left\{ \lambda e^{-\frac{\pi}{p-1}|z-z_0|^2}, A \right\}, \quad z \in \R^2,
			\end{equation}
			where $\lambda = A e^{(B/A)^p/(p-1)-1/p} > A$.
		\end{enumerate}
		Moreover, if $F$ achieves equality in \eqref{eq: estimate Nicola-Tilli first case} or \eqref{eq: estimate Nicola-Tilli second case} then $|\langle L_{F,\varphi}f,g \rangle | = \|L_{F,\varphi}\|$ for some normalized $f,g \in L^2(\R)$ if and only if both are of the kind
		\begin{equation}\label{eq: optimal function Nicola-Tilli Faber-Krahn for STFT}
			x \mapsto c \, e^{2 \pi i x \omega_0} \varphi(x-x_0), \quad x \in \R,
		\end{equation}
		 possibly with different $c$'s but with the same $(x_0, \omega_0) \in \R^{2}$ which coincides with the centre of $F$.
	\end{teo}
	We point out that optimal weight functions are not always Gaussian, indeed \eqref{eq: optimal function Nicola-Tilli second case} are \emph{Gaussians truncated above}.
	
	The aim of this paper is to find analogous estimates
	for any function $F \in L^p(\R^{2}) \cap L^q(\R^{2})$, where $p,q \in (1,+\infty)$. The main motivation is to understand the transition between the setting for Lieb's estimate ($p=q$) and that of Nicola and Tilli ($q = +\infty)$ and in particular how optimal weight functions behave in this intermediate regime. Another reason lies in the possible consequences of such optimal estimates, of which Corollary \ref{cor:estimate for spectrogram} and Corollary \ref{cor:estimate for functions in Fock space} below are an example. To state the result, we introduce the following notation:
	\begin{equation*}
		\Log(x) = \max\{-\log(x),0\}.
	\end{equation*}
	\begin{teo}\label{th:main theorem dimension 1}
		Let $p,q \in (1,+\infty)$ and $A,B \in (0, +\infty)$. Let $F \in L^p(\R^2) \cap L^q(\R^2)$ with $\|F\|_p \leq A$ and $\|F\|_q \leq B$.
		\begin{enumerate}[label=(\roman*)]
			\item If $B/A \geq \kappa_p^{(\frac{1}{q}-\frac{1}{p})} \left(\dfrac{p}{q}\right)^{\frac{1}{q}}$ $\left(\text{respectively\ } B/A \leq \kappa_q^{(\frac{1}{q}-\frac{1}{p})} \left(\dfrac{p}{q}\right)^{\frac{1}{p}}\right)$, then:
			\begin{equation}\label{eq: estimate Riccardi first case dimension 1}
				\|L_{F,\varphi}\| \leq \kappa_p^{ \kappa_p} A \quad (\text{resp.\ } \|L_{F,\varphi}\| \leq \kappa_q^{ \kappa_q} B),
			\end{equation}
			with equality if and only if, for some $\theta \in \R$ and some $z_0 \in \R^{2}$,
			\begin{equation*}\label{eq: optimal function Riccardi first case dimension 1}
				F(z) = e^{i \theta} \lambda e^{-\frac{\pi}{p-1}|z-z_0|^2} \quad (\text{resp.\ } F(z) = e^{i \theta} \lambda e^{-\frac{\pi}{q-1}|z-z_0|^2}),
			\end{equation*}
			where $\lambda = \kappa_p^{-1/p}A$ (resp. $\lambda = \kappa_q^{-1/q}B$).
			\item If $\kappa_q^{(\frac{1}{q}-\frac{1}{p})} \left(\dfrac{p}{q}\right)^{\frac{1}{p}} < B/A < \kappa_p^{(\frac{1}{q}-\frac{1}{p})} \left(\dfrac{p}{q}\right)^{\frac{1}{q}}$, then
			\begin{equation}\label{eq: estimate Riccardi second case dimension 1}
				\|L_{F,\varphi}\| \leq T - \dfrac{\lambda_1}{p} T^p - \dfrac{\lambda_2}{q} T^q
			\end{equation}
			where $u(t) = \Log (\lambda_1 t^{p-1} + \lambda_2 t^{q-1})$ and $\lambda_1, \lambda_2 > 0$ are uniquely determined by
			\begin{equation*}
				p \int_0^{+\infty} t^{p-1} u(t) \,dt = A^p, \quad q \int_0^{+\infty} t^{q-1} u(t) \,dt = B^q.
			\end{equation*}
			and $T > 0$ the unique value such that $\lambda_1 T^{p-1} + \lambda_2 T^{q-1} = 1$. Moreover, the function $t \mapsto -\log(\lambda_1 t^{p-1} + \lambda_2 t^{q-1})$ defined on $(0,T]$ is invertible and we denote by $\psi : [0, +\infty) \rightarrow (0,T]$ its inverse. Then, equality in \eqref{eq: estimate Riccardi second case dimension 1} is achieved if and only if, for some $\theta \in \R$ and some $z_0 \in \R^{2}$,
			\begin{equation}\label{eq:optimal weight function Riccardi second case dimension 1}
				F(z) = e^{i \theta} \psi(\pi |z-z_0|^2).
			\end{equation}
		\end{enumerate}
		Finally, if $F$ achieves equality in \eqref{eq: estimate Riccardi first case dimension 1} or \eqref{eq: estimate Riccardi second case dimension 1} then $|\langle L_{F, \varphi} f, g \rangle| = \|L_{F, \varphi}\|$ for some normalized $f, g \in L^2(\R)$ if and only if are both of the kind \eqref{eq: optimal function Nicola-Tilli Faber-Krahn for STFT}, possibly with different $c$'s but same $(x_0, \omega_0) \in \R^{2}$.
	\end{teo}
	One reason we decided to introduce our result in dimension 1 is that the estimate for $\|L_{F,\varphi}\|$ \eqref{eq: estimate Riccardi second case dimension 1}, is somewhat ``explicit''. In Section \ref{sec:Main result and proof} we are going to consider the problem in generic dimension $d$, where this estimate is less explicit. Also the expression of optimal weight functions is not explicit. Nevertheless, their profile can be computed numerically. In Figure \ref{fig:optimal weight function} there is an example of optimal weight function in the intermediate regime.
	
	In Section \ref{sec:Some corollaries} we are going to prove some immediate yet important consequences of Theorem \ref{th:main theorem dimension 1} and its counterpart in generic dimension, Theorem \ref{th:main theorem}. In particular,  Corollary \ref{cor:estimate for spectrogram} is an extension of Lieb's inequality \eqref{eq:Lieb's inequality} to the spaces $L^p + L^q$ with a characterization of extremal functions, while Corollary \ref{cor:estimate for functions in Fock space} is an analogous result for functions in the Fock space.
	
	Also, similar results should hold for wavelet localization operators and functions in Bergman space (see e.g. \cite{nicolatilli_norm}). We plan to investigate these issues in a subsequent work.
	
	\section{Main result and proof}\label{sec:Main result and proof}
	In this Section we are going to state and prove the analogous of Theorem \ref{th:main theorem dimension 1} in generic dimension $d$. To do so we need to introduce the following notation:
	\begin{equation}\label{eq: definition G}
		G(s) = \int_{0}^{s} e^{-(d!\tau)^{\frac{1}{d}}}\, d\tau.
	\end{equation}
	
	\begin{teo}\label{th:main theorem}
		Assume $p,q \in (1,+\infty)$, $A,B \in (0, +\infty)$ and suppose that $F \in L^p(\R^{2d}) \cap L^q(\R^{2d})$ satisfies the following constraints:
		\begin{equation}\label{eq: constraints}
			\|F\|_p \leq A, \quad \|F\|_q \leq B.
		\end{equation}
		Then
		\begin{enumerate}[label=(\roman*)]
			\item\label{th:main theorem i} If $B/A \geq \kappa_p^{d(\frac{1}{q}-\frac{1}{p})} \left(\dfrac{p}{q}\right)^{\frac{d}{q}}$ $\left(\text{respectively\ } B/A \leq \kappa_q^{d(\frac{1}{q}-\frac{1}{p})} \left(\dfrac{p}{q}\right)^{\frac{d}{p}}\right)$, then:
			\begin{equation}\label{eq:estimate Riccardi first case dimension d}
				\|L_{F,\varphi}\| \leq \kappa_p^{d \kappa_p} A \quad (\text{resp.\ } \|L_{F,\varphi}\| \leq \kappa_q^{d \kappa_q} B),
			\end{equation}
			with equality if and only if, for some $\theta \in \R$ and some $z_0 \in \R^{2d}$,
			\begin{equation*}
				F(z) = e^{i \theta} \lambda e^{-\frac{\pi}{p-1}|z-z_0|^2} \quad (\text{resp.\ } F(z) = e^{i \theta} \lambda e^{-\frac{\pi}{q-1}|z-z_0|^2}),
			\end{equation*}
			where $\lambda = \kappa_p^{-d/p}A$ (resp. $\lambda = \kappa_q^{-d/q}B$).
			\item\label{th:main theorem ii} If $\kappa_q^{d(\frac{1}{q}-\frac{1}{p})} \left(\dfrac{p}{q}\right)^{\frac{d}{p}} < B/A < \kappa_p^{d(\frac{1}{q}-\frac{1}{p})} \left(\dfrac{p}{q}\right)^{\frac{d}{q}}$, then
			\begin{equation}\label{eq:estimate Riccardi second case dimension d}
				\|L_{F,\varphi}\| \leq \int_0^{+\infty} G(u(t)) \,dt,
			\end{equation}
			where 
			\begin{equation}\label{eq:optimal distribution function}
				u(t) = \frac{1}{d!} \left[ \Log (\lambda_1 t^{p-1} + \lambda_2 t^{q-1})\right]^d
			\end{equation}
			and $\lambda_1, \lambda_2 > 0$ are uniquely determined by
			\begin{equation}\label{eq:equations for multipliers}
				p \int_0^{+\infty} t^{p-1} u(t) \,dt = A^p, \quad q \int_0^{+\infty} t^{q-1} u(t) \,dt = B^q.
			\end{equation}
			Moreover, letting $T > 0$ the unique value such that $\lambda_1 T^{p-1} + \lambda_2 T^{q-1} = 1$, the function $t \mapsto -\log(\lambda_1 t^{p-1} + \lambda_2 t^{q-1})$ defined on $(0,T]$ is invertible and we denote by $\psi : [0, + \infty) \rightarrow (0,T]$ its inverse. Then, equality in \eqref{eq:estimate Riccardi second case dimension d} is achieved if and only if, for some $\theta \in \R$ and some $z_0 \in \R^{2d}$, $F(z) = e^{i \theta} \psi(\pi |z-z_0|^2)$.
		\end{enumerate}
		Finally, if $F$ achieves equality in \eqref{eq:estimate Riccardi first case dimension d} or \eqref{eq:estimate Riccardi second case dimension d} then $|\langle L_{F, \varphi} f, g \rangle| = \|L_{F, \varphi}\|$ for some normalized $f, g \in L^2(\R^d)$ if and only if are both of the kind
		\begin{equation}\label{eq: optimal function Nicola-Tilli Faber-Krahn for STFT dimension d}
			x \mapsto c e^{2 \pi i \omega \cdot t} \varphi (x - x_0), \quad x \in \R^d,
		\end{equation}
		possibly with different $c$'s but same $(x_0, \omega_0) \in \R^{2d}$.
	\end{teo}
	We split the proof of the theorem in several steps and we start proving the first statement, which explains how these different regimes arise.
	\begin{proof}[Proof of Theorem \ref{th:main theorem}\ref{th:main theorem i}]
		We consider just the first version, since the other one can be obtained swapping $p$ and $q$.
		
		From the $d$-dimensional version of Theorem \ref{th:Lieb's inequality} (which can be found in \cite{nicolatilli_norm}) it is straightforward to see that
		\begin{equation*}
			\| L_{F, \varphi}\| \leq \min\{\kappa_p^{d\kappa_p}A, \, \kappa_q^{d\kappa_q}B\}.
		\end{equation*}
		Supposing that the first term is smaller than the second, which means:
		\begin{equation}\label{eq:condition B/A first}
			\kappa_p^{d\kappa_p}A \leq \kappa_q^{d\kappa_q}B \iff \dfrac{B}{A} \geq \left(\dfrac{\kappa_p^{\kappa_p}}{\kappa_q^{\kappa_q}}\right)^d,
		\end{equation}
		we want to check whether an optimal weight function for the problem with just the $L^p$ constraint, namely:
		\begin{equation*}
			F(z) = A \kappa_p^{-\frac{d}{p}} e^{i \theta} e^{-\frac{\pi}{p-1}|z-z_0|^2}, \quad z \in \R^{2d},
		\end{equation*}
		satisfies also the $L^q$ constraint. From a direct computation one obtains
		\begin{align*}
			\| F \|_q &= A \kappa_p^{-\frac{d}{p}} \left(\dfrac{p-1}{q}\right)^{\frac{d}{q}},
		\end{align*}
		hence $F$ satisfies also the $L^q$ constraint if and only if $\|F\|_q \leq A$, which is equivalent to
		\begin{equation}\label{eq:condition B/A second}
			\dfrac{B}{A} \geq \kappa_p^{d\left(\frac1q - \frac1p\right)}\left(\dfrac{p}{q}\right)^{\frac{d}{q}}.
		\end{equation}
		If this condition is met, the solution with just the $L^p$ constraint is a solution also for the problem with two constraints. Moreover, from Theorem \ref{th:Nicola-Tilli norm} follows also the last part of the statement regarding those $f$ and $g$ that achieve equality in $|\langle L_{F,\varphi} f, g \rangle| = \|L_{F,\varphi}\|$.
	\end{proof}
	If condition \eqref{eq:condition B/A second} were less restrictive than condition \eqref{eq:condition B/A first} we would have completely solved the problem. Unfortunately, this is not the case. Indeed it is always true, regardless of $p$ and $q$, that
	\begin{equation}\label{curious inequality between conjugate exponents original}
		\kappa_p^{d\left(\frac1q - \frac1p\right)}\left(\dfrac{p}{q}\right)^{\frac{d}{q}} \geq \left(\dfrac{\kappa_p^{\kappa_p}}{\kappa_q^{\kappa_q}}\right)^d.
	\end{equation}
	We mention that this inequality can be stated in a rather curious way:
	\begin{equation*}
		\left(\dfrac{p'}{q'}\right)^{\frac1{q'}} \left(\dfrac{p}{q}\right)^{\frac{1}{q}} \geq 1.
	\end{equation*}
	
	Therefore, if $B/A \geq \kappa_p^{d\left(\frac{1}{q}-\frac{1}{p}\right)}\left(\frac{p}{q}\right)^{\frac{d}{q}}$ or $B/A \leq \kappa_q^{d\left(\frac{1}{q}-\frac{1}{p}\right)}\left(\frac{p}{q}\right)^{\frac{d}{p}}$, the problem is already solved and the solution is given by Theorem \ref{th:Lieb's inequality}. Therefore, from now on, we will consider the intermediate case, that is:
	\begin{equation}\label{intermediate regime}
		\kappa_q^{d\left(\frac{1}{q}-\frac{1}{p}\right)} \left(\dfrac{p}{q}\right)^{\frac{d}{p}} < \dfrac{B}{A} < \kappa_p^{d\left(\frac{1}{q}-\frac{1}{p}\right)} \left(\dfrac{p}{q}\right)^{\frac{d}{q}},
	\end{equation}
	which corresponds to the statement of \ref{th:main theorem}\ref{th:main theorem ii}. We notice that the condition is well-posed, since it is actually true that
	\begin{align}\label{another inequality between conjugate exponents original}
		\kappa_q^{d\left(\frac{1}{q}-\frac{1}{p}\right)} \left(\dfrac{p}{q}\right)^{\frac{d}{p}} < \kappa_p^{d\left(\frac{1}{q}-\frac{1}{p}\right)} \left(\dfrac{p}{q}\right)^{\frac{d}{q}}
	\end{align}
	whenever $p \neq q$.
	
	The starting point to prove \ref{th:main theorem}\ref{th:main theorem ii} is a Theorem from \cite{nicolatilli_norm} which gives a bound for $\| L_{F, \varphi} \|$ in terms of the distribution function of $|F|$.
	\begin{teo}\label{th: norm limitation}
		Assume $F \in L^p(\R^{2d})$ for some $p \in [1,+\infty)$ and let $\mu(t) = |\{|F|>t\}|$ be the distribution function of $|F|$. Then
		\begin{equation}\label{norm limitation formula}
			\| L_{F, \varphi} \| \leq \int_0^{+\infty} G(\mu(t))\,dt.
		\end{equation}
		Equality occurs if and only if $F(z)=e^{i\theta}\rho(|z-z_0|)$ for some $\theta \in \R$, $z_0 \in \R^{2d}$ and some nonincreasing function $\rho : [0,+\infty) \rightarrow [0,+\infty)$. In this case, it holds $|\langle L_{F,\varphi}f,g \rangle| = \|L_{F,\varphi}\|$ for some normalized Gaussians $f$ and $g$ of the kind \eqref{eq: optimal function Nicola-Tilli Faber-Krahn for STFT dimension d}, possibly with different $c$'s but with same centre $z_0 = (x_0,\omega_0) \in ~ \R^{2d}$.
	\end{teo}
	
	In light of the previous Theorem, it is natural to seek for a sharp upper bound for the right-hand side of \eqref{norm limitation formula}. Since this involves the distribution function $|F|$, we shall search this bound between all the possible distribution functions. In order to do so, we need to rephrase constraints \eqref{eq: constraints} in terms of $\mu$. This can be easily done thanks to the ``layer cake'' representation (see, for example, \cite[][Theorem 1.13]{liebloss}):
	\begin{equation*}
		\|F\|_p^p = p \int_0^{+\infty} t^{p-1}|\{|F|>t\}|\,dt.
	\end{equation*}
	Hence, constraints \eqref{eq: constraints} become
	\begin{equation}\label{eq: constraints distribution function}
		p \int_0^{+\infty} t^{p-1} |\{|F|>t\}| \,dt \leq A^p \quad \text{and} \quad q \int_0^{+\infty} t^{q-1} |\{|F|>t\}| \,dt \leq B^q
	\end{equation}
	and we can define the proper space of possible distribution functions
	\begin{equation}\label{eq: distribution function space}
		\mathcal{C} = \{u : (0,+\infty) \rightarrow [0,+\infty) \text{ such that } u \text{ is decreasing and satisfies } \eqref{eq: constraints distribution function}\}.
	\end{equation}
	We have reached the point where our original question is rephrased in the following variational problem:
	\begin{equation}\label{nonstandard variational problem formulation}
		\sup_{v \in \mathcal{C}} I(v) \quad \text{where} \quad I(v) \coloneqq \int_0^{+\infty} G(v(t))\,dt.
	\end{equation}
	Firstly, we shall prove existence of maximizers. The proof closely follows the one made in \cite{nicolatilli_norm}, but the result is slightly different.
	\begin{prop}\label{prop:existence of maximizer}
		The supremum in \eqref{nonstandard variational problem formulation} is finite and it is attained by at least one function $u \in \mathcal{C}$. Moreover, every extremal function $u$ achieves equality in at least one of the constraints \eqref{eq: constraints distribution function}.
	\end{prop}
	\begin{proof}
		Considering, for example, the first constraint in \eqref{eq: constraints distribution function}, we see that
		\begin{equation*}
			t^p u(t) = p \int_0^t \tau^{p-1} u(t) \,d\tau \overset{u \text{ decreasing}}{\leq} p \int_0^t \tau^{p-1} u(\tau) \,d\tau \leq A^p,
		\end{equation*}
		hence functions in $\mathcal{C}$ are pointwise bounded by $A^p/t^p$. It is straightforward to verify that $G$ in \eqref{eq: definition G} is increasing, that $G(s) \leq s$ and that $G(s) \leq 1$. Using these properties we have:
		\begin{align*}
			I(u) &= \int_0^{+\infty} G(u(t))\,dt = \int_0^1 G(u(t))\,dt + \int_1^{+\infty} G(u(t))\,dt \overset{G(s) \leq 1}{\leq} 1 + \int_1^{+\infty} G(u(t))\,dt \\
			\overset{G \text{ increasing}}&{\leq} 1 + \int_1^{+\infty} G(A^p/t^p)\,dt \overset{G(s) \leq s}{\leq} 1 + \int_1^{+\infty} \dfrac{A^p}{t^p}\,dt < +\infty,
		\end{align*}
		therefore the supremum in \eqref{nonstandard variational problem formulation} is finite.
		
		Let $\{u_n\}_{n \in \N} \subset \mathcal{C}$ be a maximizing sequence. Since every $u_n$ is pointwise bounded by $A^p/t^p$, thanks to Helly's selection theorem we can say that, up to a subsequence, $u_n$ converges pointwise to a decreasing function $u$. Moreover, $u$ is still in $\mathcal{C}$, indeed:
		\begin{equation*}
			\int_0^{+\infty} t^{p-1} u(t) = \int_0^{+\infty} \lim_{n \rightarrow +\infty} t^{p-1} u_n(t) \,dt \overset{\text{Fatou's lemma}}{\leq} \liminf_{n \rightarrow +\infty} \int_0^{+\infty} t^{p-1} u_n(t) \,dt \leq \dfrac{A^p}{p},
		\end{equation*}
		and clearly the same holds for $q$ instead of $p$.
		
		Now we have to prove that $u$ is actually achieving the supremum. We already saw that the following holds:
		\begin{equation*}
			|G(u_n(t))| \leq \chi_{(0,1)}(t) + \dfrac{A^p}{t^p} \chi_{(1,+\infty)}(t)
		\end{equation*}
		and that the left-hand side is a function in $L^1(0,+\infty)$. This allows us to use dominated convergence theorem to conclude that 
		\begin{equation*}
			I(u) = \int_0^{+\infty} G(u(t)) = \lim_{n \rightarrow +\infty} \int_0^{+\infty} G(u_n(t))\,dt = \lim_{n \rightarrow +\infty} I(u_n) = \sup_{v \in \mathcal{C}} I(v).
		\end{equation*}
		
		Lastly, we need to show that $u$ achieves equality at least in one of the constraints \eqref{eq: constraints distribution function}. Supposing that this is not true, if we let $u_{\varepsilon}(t) = (1+\varepsilon) u(t)$, then for $\varepsilon > 0$ sufficiently small constraints are still satisfied and since $G$ is strictly increasing $I(u_{\varepsilon}) > I(u)$, which contradicts the hypothesis that $u$ is a maximizer.
	\end{proof}
	After proving the existence of maximizers, we shall prove that removing the monotonicity assumption in \eqref{eq: distribution function space} does not change the solution of the variational problem.
	\begin{prop}\label{monotonicity of maximizer}
		Let $\allowdisplaybreaks[1]\mathcal{C}' = \{u : (0,+\infty) \rightarrow [0,+\infty)$ such that  $u$ is measurable and satisfies \eqref{eq: constraints distribution function}$\}$. Then
		\begin{equation}
			\sup_{v \in \mathcal{C}} I(v) = \sup_{v \in \mathcal{C}'} I(v).
		\end{equation}
		In particular, any function $u \in \mathcal{C}$ achieving the supremum on the left-hand side also achieves it on the right-hand side.
	\end{prop}
	\begin{proof}
		Let $u \in \mathcal{C}'$. We define its \emph{decreasing rearrangement} as:
		\begin{equation}
			u^*(s) = \sup\{t \geq 0 : |\{u>t\}|>s\},
		\end{equation}
		with the convention that $\sup \emptyset = 0$. It is clear from the definition that $u^*$ is a non-increasing function. Moreover, one can see that $u^*$ is right-continuous and that $u$ and $u^*$ are \emph{equi-measurable} (\cite[][Section 10.12]{hardy_littlewood_polya}, \cite[][Proposition 1.4.5]{grafakos}), which means that they have the same distribution function. Moreover, we already pointed out that constraints \eqref{eq: constraints distribution function} imply that $u$ is pointwise bounded by $A^p/t^p$, therefore $u^*$ takes only finite values. Our aim is to show that $u^* \in \mathcal{C}$. Letting $\nu$ be the Radon measure with density $t^{p-1}$, we start proving that $\nu(\{u>s\}) \geq \nu(\{u^* > s\})$, indeed:
		\begin{align*}
			\nu(\{u>s\}) &= \int_{\{u>s\}} t^{p-1} \,dt \overset{t^{p-1}\; \mathrm{increasing}}{\geq} \int_0^{|\{u>s\}|} t^{p-1} \,dt  \\
			\overset{\mathrm{equi-measurability}}&{=} \int_0^{|\{u^*>s\}|} t^{p-1} \,dt \overset{u^*\; \mathrm{decreasing}}{\underset{\mathrm{right-continuous}}{=}} \int_{\{u^*>s\}} t^{p-1} \,dt = \nu(\{u^* > s\}).
		\end{align*}
		Then, using one more time the ``layer cake'' representation:
		\begin{align*}
			\int_{0}^{+\infty} t^{p-1} u(t) \,dt &= \int_{0}^{+\infty} u(t) d\nu(t) = \int_{0}^{+\infty} \nu(\{u>s\}) \,ds \geq \\
			&= \int_{0}^{+\infty} \nu(\{u^* > s\}) \,ds = \int_{0}^{+\infty} u^*(t) d\nu(t) = \int_{0}^{+\infty} t^{p-1} u^*(t) \,dt.
		\end{align*}
		If we swap $p$ with $q$ we conclude that $u^* \in \mathcal{C}$. Moreover, always from equi-measurability, we have:
		\begin{align*}
			I(u) &= \int_{0}^{+\infty} G(u(t)) \,dt = \int_{0}^{+\infty} \int_0^{u(t)} e^{-(d!\tau)^{1/d}} \,d\tau \,dt = \int_{0}^{+\infty} \int_{0}^{+\infty} \chi_{\{u>\tau\}}(t) e^{-(d!\tau)^{1/d}} \,d\tau \,dt \\
			\overset{\mathrm{Tonelli}}&{=}  \int_{0}^{+\infty} |\{u > \tau\}| e^{-(d!\tau)^{1/d}} \,d\tau = \int_{0}^{+\infty} |\{u^* > \tau\}| e^{-(d!\tau)^{1/d}} \,d\tau = I(u^*).
		\end{align*}
		Taking the supremum over all possible $u \in \mathcal{C'}$ we have:
		\begin{align*}
			\sup_{v \in \mathcal{C}'} I(v)= \sup_{v \in \mathcal{C}'} I(v^*) \leq \sup_{v \in \mathcal{C}} I(v).
		\end{align*}
		Inequality $\sup_{v \in \mathcal{C}'} I(v) \geq \sup_{v \in \mathcal{C}} I(v)$ is trivial since $\mathcal{C}' \supset \mathcal{C}$.
		
	\end{proof}
	We are now in the position to find maximizers of \eqref{nonstandard variational problem formulation}.
	\begin{teo}\label{nonstandard variational problem solution theorem}
		There exist a unique function $u \in \mathcal{C}$ achieving the supremum in \eqref{nonstandard variational problem formulation} that is:
		\begin{equation}\label{nonstandard variational problem solution formula}
			u(t) = \dfrac{1}{d!} \left[\Log\left(\lambda_1 t^{p-1} + \lambda_2 t^{q-1}\right)\right]^d, \quad t>0
		\end{equation}
		where $\lambda_1, \lambda_2$ are both positive and uniquely determined by
		\begin{equation*}
			p\int_0^{+\infty} t^{p-1}u(t)\,dt = A^p, \quad q\int_0^{+\infty} t^{q-1}u(t)\,dt = B^q.
		\end{equation*}
	\end{teo}
	\begin{proof}
		We will split the proof in several parts. Firstly we will show that maximizers are given by \eqref{nonstandard variational problem solution formula}. Then we will show that multipliers $\lambda_1$ and $\lambda_2$ are both strictly positive and unique.
		
		\textit{Expression of maximizers}: Let $M = \sup\{t \in (0,+\infty)\ : u(t) > 0\}$. From Proposition \ref{prop:existence of maximizer} we know that $u$ has to achieve at least one of the constraints, therefore $M>0$. Consider now a closed interval $[a,b] \subset (0,M)$ and a function $\eta \in L^{\infty}(0,M)$ supported in $[a,b]$. Without loss of generality we can suppose that $\eta$ is orthogonal, in the $L^2$ sense, to $t^{p-1}$ and $t^{q-1}$, explicitly
		\begin{equation}\label{variations are orthogonal to constraints}
			\int_a^b t^{p-1} \eta(t) \,dt=0, \quad \int_a^b t^{q-1} \eta(t) \,dt=0.
		\end{equation}
		On $[a,b]$ we have that $u(t) \geq u(b) > 0$, hence, for $|\varepsilon|$ sufficiently small, $u+\varepsilon\eta$ is still a nonnegative function which satisfies \eqref{eq: constraints distribution function}, therefore $u+\varepsilon\eta \in \mathcal{C}'$. Since we are supposing that $u$ is a maximizer, the function $\varepsilon \mapsto I(u+\varepsilon\eta)$ has a maximum for $\varepsilon = 0$. Given that $\eta$ is supported in a compact interval we can differentiate under the integral sign and obtain
		\begin{equation*}
			0 = \dfrac{d}{d\varepsilon}I(u+\varepsilon\eta) \lvert_{\varepsilon=0} = \int_a^b G'(u(t))\eta(t)\,dt.
		\end{equation*}
		We would like to extend this result to every $\eta$ in $L^2(a,b)$ satisfying \eqref{variations are orthogonal to constraints}. Since $L^{\infty}(a,b)$ is dense in $L^2(a,b)$, there exist a sequence $\{\eta_k\}_{k \in \N} \subset L^{\infty}(a,b)$ such that $\eta_k \rightarrow \eta$ in $L^2(a,b)$. We can consider the projection operator $P$ such that, given $\psi \in L^2(a,b)$, $P\psi$ is the orthogonal projection of $\psi$ onto $ X = \mathrm{span}\{t^{p-1},t^{q-1}\}^{\perp} \subset L^2(a,b)$. Considering $P$ is continuous we have that $P\eta_k \rightarrow P\eta = \eta$, hence, since $P\eta_k \in L^{\infty}(a,b)$:
		\begin{equation*}
			0 = \int_a^b G'(u(t)) P\eta_k(t) \,dt = \langle G'(u), P\eta_n \rangle_{L^2(a,b)} \rightarrow \langle G'(u), \eta \rangle_{L^2(a,b)} = \int_a^b G'(u(t)) \eta(t) \,dt
		\end{equation*}
		namely
		\begin{equation}\label{orthogonality of G'}
			\int_a^b G'(u(t)) \eta(t) \,dt = 0. 
		\end{equation}
		Since \eqref{orthogonality of G'} holds for every $\eta \in X$ it must be that 
		\begin{equation*}
			G'(u) \in X^{\perp} = \left(\mathrm{span}\{t^{p-1},t^{q-1}\}^{\perp}\right)^{\perp} = \mathrm{span}\{t^{p-1},t^{q-1}\} \quad \text{in } (a,b).
		\end{equation*}
		Letting $a \rightarrow 0^+$ and $b \rightarrow M^-$ we then obtain
		\begin{equation}\label{expression G'(u)}
			G'(u(t)) = \lambda_1 t^{p-1} + \lambda_2 t^{q-1} \quad \text{for a.e. } t \in (0,M)
		\end{equation}
		for some multipliers $\lambda_1,\lambda_2 \in \R$. Since $u$ is decreasing actually \eqref{expression G'(u)} holds for every $t \in (0,M)$. Finally, recalling the expression of \eqref{eq: definition G} we see that $G'(s) = e^{-(d!s)^{1/d}}$. Since $u$ is monotonically decreasing we can invert \eqref{expression G'(u)} thus obtaining the explicit expression of maximizers:
		\begin{equation}\label{expression u}
			u(t) = \begin{cases}
				\dfrac{1}{d!} \left[-\log\left(\lambda_1 t^{p-1} + \lambda_2 t^{q-1}\right) \right]^d & t \in (0,M)\\
				0 & t \in (M,+\infty)
			\end{cases}
		\end{equation}
		We remark that a priori it was possible that $M=+\infty$, but from the explicit expression of maximizers we see that this is not possible since $u$ has to be nonnegative.
		
		\textit{Maximizers achieve equality in both constraints and multipliers are non-zero}: The argument we used to determine the expression of maximizers enables us to say that these have to achieve equality in both constraints in \eqref{eq: constraints distribution function}. Indeed, if, for example, we had that $q \int_0^{+\infty} t^{q-1}u(t)\,dt < B^q$, the second condition of orthogonality in \eqref{variations are orthogonal to constraints} could be removed, because for sufficiently small $\varepsilon$ a variation non-orthogonal to $t^{q-1}$ would be admissible. This would provide us the solution of the same variational problem but without the $L^q$ constraint. Since we are working in the intermediate case, we know that actually this solution does not satisfy the $L^q$ constraint, hence we conclude that $u$ has to achieve equality in both constraints. With the very same reasoning we can say that neither $\lambda_1$ nor $\lambda_2$ can be 0.
		
		\textit{Multipliers are positive}: Suppose that one of the multipliers, for example $\lambda_2$, is negative. Consider an interval $[a,b] \subset (0,M)$ and an admissible variation $\eta \in L^{\infty}(0,M)$ supported in $[a,b]$ and such that $\int_{a}^{b}t^{q-1}\eta(t)\,dt <0$. An example of such variation can be
		\begin{equation*}
			\eta(t) = \left\{
			\begin{aligned}
				-t^{1-p},\quad & t \in [a, (a+b)/2)\\
				t^{1-p},\quad					   & t \in [(a+b)/2, b]
			\end{aligned}\right.
		\end{equation*}
		if $q < p$ or
		\begin{equation*}
			\eta(t) = \left\{
			\begin{aligned}
				t^{1-p},\quad & t \in [a, (a+b)/2)\\
				-t^{1-p},\quad					   & t \in [(a+b)/2, b]
			\end{aligned}\right.
		\end{equation*}
		if $q>p$. Then the directional derivative of $G$ at $u$ along $\eta$ is:
		\begin{equation*}
			\int_{a}^{b} G'(u(t))\eta(t)\,dt = \int_{a}^{b} (\lambda_1 t^{p-1} + \lambda_2 t^{q-1})\eta(t)\,dt = \lambda_2 \int_{a}^{b} t^{q-1} \eta(t)\,dt > 0,
		\end{equation*}
		which contradicts the fact that $u$ is a maximizer.
		
		\textit{$u$ is continuous}: Now that we now that both multipliers are positive we can prove that $u$ is continuous, which is equivalent to say that $M=T$, where $T$ is the unique positive number such that $\lambda_1 T^{p-1} + \lambda_2 T^{q-1} = 1$ (uniqueness of $T$ follows from the positivity of multipliers).
		
		We start supposing that $M < T$, which means that $\lim_{t \rightarrow M^-} u(t) > 0$. Consider the following variation 
		\begin{equation*}
			\eta(t) = \left\{
			\begin{aligned}
				-1 + \alpha \frac{t}{M} + \beta,\quad & t \in (M-M\delta,M)\\
				1,\quad					   & t \in (M,M+M\delta)\\
				0,\quad				   & \text{otherwise}
			\end{aligned}\right.
		\end{equation*}
		where $\delta>0$ is small enough so that $M-M\delta >0$ and $M+M\delta < T$, while $\alpha$ and $\beta$ are constants, depending on $\delta$, to be determined. Since we want this to be an admissible variation, we impose that $\eta$ is orthogonal to $t^{p-1}$ and $t^{q-1}$. For example, the first condition is:
		\begin{align*}
			0 &= \int_{M-M\delta}^{M+M\delta} t^{p-1} \eta(t) \,dt = -\int_{M-M\delta}^M t^{p-1}\,dt + \int_{M-M\delta}^M t^{p-1}\left(\alpha \frac{t}{M} + \beta\right) \,dt + \int_M^{M+M\delta} t^{p-1}\,dt \\
			\overset{\tau=t/M}&{=}M^p \int_{1-\delta}^1 \tau^{p-1}(\alpha \tau + \beta) \,d\tau - M^p \int_{1-\delta}^1 \tau^{p-1} \,d\tau + M^p \int_1^{1+\delta} \tau^{p-1} \,d\tau
		\end{align*}
		therefore, dividing by $\delta$:
		\begin{align*}
			\fint_{1-\delta}^1 \tau^{p-1}(\alpha \tau + \beta) \,d\tau = \alpha \fint_{1-\delta}^1 \tau^{p} \,d\tau + \beta \fint_{1-\delta}^1 \tau^{p-1} \,d\tau = \fint_{1-\delta}^1 \tau^{p-1} \,d\tau - \fint_1^{1+\delta} \tau^{p-1} \,d\tau.
		\end{align*}
		The equation stemming from the orthogonality with $t^{q-1}$ is analogous. Therefore, we obtained a nonhomogeneous linear system for $\alpha$ and $\beta$:
		\begin{equation}\label{system continuity}
			\begin{pmatrix}
				\fint_{1-\delta}^1 \tau^{p} \,d\tau &   \fint_{1-\delta}^1 \tau^{p-1} \,d\tau\\
				\fint_{1-\delta}^1 \tau^{q} \,d\tau &   \fint_{1-\delta}^1 \tau^{q-1} \,d\tau
			\end{pmatrix}
			\begin{pmatrix}
				\alpha\\
				\beta
			\end{pmatrix}=
			\begin{pmatrix}
				\fint_{1-\delta}^1 \tau^{p-1} \,d\tau - \fint_1^{1+\delta} \tau^{p-1} \,d\tau\\
				\fint_{1-\delta}^1 \tau^{q-1} \,d\tau - \fint_1^{1+\delta} \tau^{q-1} \,d\tau
			\end{pmatrix}.
		\end{equation}
		This system has a unique solution if and only if the determinant of the matrix is not 0. We can show this directly:
		\begin{align*}
			& \fint_{1-\delta}^1 \tau^{p} \,d\tau \fint_{1-\delta}^1 \tau^{q-1} \,d\tau - \fint_{1-\delta}^1 \tau^{q} \,d\tau \fint_{1-\delta}^1 \tau^{p-1} \,d\tau=\\
			&= \dfrac{1}{\delta^2} \int_{(1-\delta,1)^2} \left(\tau^p \sigma^{q-1} - \tau^{p-1}\sigma^q\right) \,d\tau \,d\sigma = \dfrac{1}{\delta^2} \int_{(1-\delta,1)^2} \tau^{p-1} \sigma^{q-1} \left( \tau - \sigma \right) \,d\tau \,d\sigma = \\
			&= \dfrac{1}{\delta^2} \left( \int_{Q_1} \tau^{p-1} \sigma^{q-1} \left( \tau - \sigma \right) \,d\tau \,d\sigma + \int_{Q_2} \tau^{p-1} \sigma^{q-1} \left( \tau - \sigma \right) \,d\tau \,d\sigma \right),
		\end{align*}
		where $Q_1=(1-\delta,1)^2 \cap \{\tau > \sigma\}$ and $Q_2=(1-\delta,1)^2 \cap \{\tau < \sigma\}$. In the second integral we can consider the change of variable that swaps $\tau$ and $\sigma$. In this case, the new domain is $Q_1$, hence:
		\begin{align*}
			&\fint_{1-\delta}^1 \tau^{p} \,d\tau \fint_{1-\delta}^1 \tau^{q-1} \,d\tau - \fint_{1-\delta}^1 \tau^{q} \,d\tau \fint_{1-\delta}^1 \tau^{p-1} \,d\tau\\
			&= \dfrac{1}{\delta^2} \int_{Q_1} \left(\tau^{p-1}\sigma^{q-1} - \tau^{q-1}\sigma^{p-1}\right) \left(\tau - \sigma\right) \,d\tau \,d\sigma.
		\end{align*}
		In $Q_1$ we have that $\tau - \sigma > 0$ and the sign of $\tau^{p-1}\sigma^{q-1} - \tau^{q-1}\sigma^{p-1}$ is constant, indeed:
		\begin{equation*}
			\tau^{p-1}\sigma^{q-1} - \tau^{q-1}\sigma^{p-1} > 0 \iff \left(\dfrac{\tau}{\sigma}\right)^{p-q} > 1 \overset{\tau > \sigma}{\iff} p>q.
		\end{equation*}
		Therefore the determinant of the matrix is always not 0.
		
		Now that we have an admissible variation, we can compute the directional derivative of $G$ along $\eta$. Since $u$ is supposed to be a maximizer, this derivative has to be nonpositive, therefore:
		\begin{align*}
			0 &\geq \int_{M-M\delta}^{M+M\delta} G'(u(t))\eta(t)\,dt = -\int_{M-M\delta}^M \left(\lambda_1 t^{p-1} + \lambda_2 t^{q-1}\right) \,dt\, +\\
			&+ \int_{M-M\delta}^M \left(\lambda_1 t^{p-1} + \lambda_2 t^{q-1}\right)\left(\alpha \dfrac{t}{M}+\beta\right) \,dt + \int_M^{M+M\delta}\,dt = \\
			&= -\int_{M-M\delta}^M \left(\lambda_1 t^{p-1} + \lambda_2 t^{q-1}\right) \,dt + \lambda_1 M^p \int_{1-\delta}^1 t^{p-1}(\alpha t + \beta) \,dt +\\
			&+ \lambda_2 M^q \int_{1-\delta}^1 t^{q-1}(\alpha t + \beta) \,dt + M\delta.
		\end{align*}
		Dividing by $M\delta$ and rearranging we obtain:
		\begin{equation}\label{inequality continuity}
			\begin{aligned}
				\fint_{M-M\delta}^M \left(\lambda_1 t^{p-1} + \lambda_2 t^{q-1}\right) \,dt \geq  1 &+ \lambda_1 M^{p-1} \fint_{1-\delta}^1 t^{p-1}(\alpha t + \beta) \,dt\\ &+ \lambda_2 M^{q-1} \fint_{1-\delta}^1 t^{q-1}(\alpha t + \beta) \,dt.
			\end{aligned}
		\end{equation}
		We notice that the last two terms are exactly the ones that appear in the orthogonality condition, therefore, to understand their behavior as $\delta$ approaches 0, we need to study the right-hand side of the system \eqref{system continuity}. If we expand the first component in the right-hand side of \eqref{system continuity} in its Taylor series with respect to $\delta$ we have:
		\begin{align*}
			\left(1-\dfrac{p-1}{2}\delta + o(\delta) \right) - \left(1+\dfrac{p-1}{2}\delta + o(\delta) \right) = -(p-1)\delta + o(\delta)
		\end{align*}
		and similarly for the other component. If we let $\delta \rightarrow 0^+$ in \eqref{inequality continuity} we obtain
		\begin{align*}
			\lambda_1 M^{p-1} + \lambda_2 M^{q-1} &= \lim_{\delta \rightarrow 0^+} \fint_{M-M\delta}^M \left(\lambda_1 t^{p-1} + \lambda_2 t^{q-1}\right) \,dt \\
			&\geq 1 + \lim_{\delta \rightarrow 0^+} \left[ \lambda_1 M^{p-1} \fint_{1-\delta}^1 t^{p-1}(\alpha t + \beta) \,dt + \lambda_2 M^{q-1} \fint_{1-\delta}^1 t^{q-1}(\alpha t + \beta) \right] \\
			&= 1 + \lambda_1 M^{p-1} \lim_{\delta \rightarrow 0^+} \left[ -(p-1)\delta + o(\delta) \right] + \lambda_2 M^{q-1} \lim_{\delta \rightarrow 0^+} \left[ -(q-1)\delta + o(\delta) \right]\\
			&= 1.
		\end{align*}
		The function $\lambda_1 t^{p-1}  + \lambda_2 t^{q-1}$ is strictly increasing because $\lambda_1$ and $\lambda_2$ are both positive, therefore this implies that $M \geq T$, which is absurd because we supposed that $M<T$. This allows us to write $u$ as in \eqref{nonstandard variational problem solution formula}.
		
		\textit{Uniqueness of multipliers}:	Lastly we shall prove that multipliers $\lambda_1, \lambda_2$, and hence maximizer, are unique. For this proof it is convenient to express $u$ in a slightly different way:
		\begin{equation*}
			u(t) = \dfrac{1}{d!}\left[ \Log\left((c_1t)^{p-1} + (c_2t)^{q-1}\right) \right]^d.
		\end{equation*} 
		To emphasize that $u$ is parametrized by $c_1, c_2$ we write $u(t;c_1,c_2)$. If we let $u$ into \eqref{eq: constraints distribution function} we obtain the following functions of $c_1$ and $c_2$:
		\begin{equation*}
			f(c_1,c_2) = p\ \int_0^T t^{p-1}u(t;c_1,c_2)\,dt, \quad g(c_1,c_2) = q\ \int_0^T t^{q-1}u(t;c_1,c_2)\,dt.
		\end{equation*}
		We want to highlight that, even if it is not explicit, also $T$ depends on $c_1$ and $c_2$. Nevertheless, these functions are differentiable since both $T$ and $u$ are differentiable with respect to $(c_1,c_2)$ and $t^{p-1}u$, $t^{q-1}u$ and their derivatives are bounded in $(0,T)$.\\
		Our maximizer $u$ satisfies the constraints only if $f(c_1,c_2)=A^p$ and $g(c_1,c_2) = B^q$, therefore uniqueness of the maximizer is equivalent to the fact that level sets $\{f=A^p\}$ and $\{g=B^q\}$ intersect only in one point.
		
		The first step is to consider the intersection of level sets with the coordinate axis, namely when one of $c_1$ or $c_2$ is 0. For example, if $c_2=0$ for $f$ we obtain:
		\begin{equation*}
			f(c_{1,f}, 0) = A^p \implies c_{1,f} = \dfrac{\kappa_p^{\frac{d}{p}}}{A}.
		\end{equation*}
		The same can be done for $g$ and setting $c_1=0$ instead of $c_2 = 0$. Thus, we obtain four points:
		\begin{equation*}
			c_{1,f} = \dfrac{\kappa_p^{\frac{d}{p}}}{A},\ c_{1,g} = \left(\dfrac{p-1}{q}\right)^{\frac{d}{q}}\dfrac1{B},\ c_{2,f} = \left(\dfrac{q-1}{p}\right)^{\frac{d}{p}}\dfrac1{A},\ c_{2,g} = \dfrac{\kappa_q^{\frac{d}{q}}}{B}.
		\end{equation*}
		In the regime we are considering one has that $c_{1,f} < c_{1,g}$ and $c_{2,f} > c_{2,g}$, indeed:
		\begin{align*}
			&c_{1,f} < c_{1,g} \iff \dfrac{B}{A} < \kappa_p^{d\left( \frac1{q}-\frac1{p}\right)}\left(\dfrac{p}{q}\right)^{\frac{d}{q}},\\
			&c_{2,f} > c_{2,g} \iff \dfrac{B}{A} > \kappa_q^{d\left( \frac1{q}-\frac1{p}\right)}\left(\dfrac{p}{q}\right)^{\frac{d}{p}},
		\end{align*}
		which are exactly conditions in \eqref{intermediate regime}.
		
		Then, we notice that, for every value $c_1 \in (0,c_{1,f})$, there exists a unique value of $c_2$ for which $f(c_1,c_2) = A^p$. Indeed, from previous computations we notice that $f(c_1,0)$ is a decreasing function hence $f(c_1,0) > A^p$, while $\lim_{c_2 \rightarrow +\infty} f(c_1,c_2) = 0$, therefore from the intermediate value theorem it follows that $f(c_1,c_2) = A^p$ for some $c_2$. The uniqueness of this value follows from strict monotonicity of $f(c_1,\cdot)$, indeed:
		\begin{equation}\label{df/dc1}
			\pfrac{f}{c_1}(c_1,c_2) = -\dfrac{p(p-1)}{(d-1)!}c_1^{p-2}\ \int_0^T \dfrac{t^{2(p-1)}}{(c_1t)^{p-1}+(c_2t)^{q-1}} \left[ -\log\left((c_1t)^{p-1} + (c_2t)^{q-1}\right) \right]^{d-1}\,dt,
		\end{equation}
		is always strictly negative. We point out that the term $\frac{\partial T}{\partial c_1}(c_1,c_2)u(T;c_1,c_2)$, that should appear since $T$ depends on $c_1$, is 0 because $u$ is 0 in $T$. The same is true for $g$, therefore on the interval $(0,c_{1,f})$ the level sets of $f$ and $g$ can be seen as the graph of two functions we denote with $\phi$ and $\gamma$. Moreover, since $f$ and $g$ are both differentiable, from the implicit function theorem we have that $\phi$ and  $\gamma$ are differentiable with respect to $c_1$.
		
		After defining $\phi$ and $\gamma$ we want to study their difference $\phi - \gamma$. First of all, from the intermediate zero theorem we have that $\phi - \gamma$ vanishes in at least one point, since
		\begin{equation*}
			(\phi-\gamma)(0) = c_{2,f} - c_{2,g} > 0, \quad (\phi - \gamma)(c_{1,f}) = -\gamma(c_{1,f}) < 0.
		\end{equation*}
		Then, to prove the uniqueness of the zero we will show that $(\phi-\gamma)' < 0$ whenever $\phi=\gamma$. Again by the implicit function theorem we have
		\begin{equation*}
			\begin{split}
				\dfrac{d}{d c_1}(\phi - \gamma)(c_1) = -\dfrac{\pfrac{f}{c_1}(c_1,\phi(c_1))}{\pfrac{f}{c_2}(c_1,\phi(c_1))}  + 
				\dfrac{\pfrac{g}{c_1}(c_1,\gamma(c_1))}{\pfrac{g}{c_2}(c_1,\gamma(c_1))} < 0 \iff \\
				\mathcal{I}(c_1) = \pfrac{f}{c_1}(c_1,\phi(c_1)) \pfrac{g}{c_2}(c_1,\gamma(c_1)) - \pfrac{f}{c_2}(c_1,\phi(c_1)) \pfrac{g}{c_1}(c_1,\gamma(c_1)) > 0
			\end{split}
		\end{equation*}
		As for \eqref{df/dc1} the other derivatives are computed. To simplify the notation we let 
		\begin{equation*}
			h(t;c_1,c_2) = \frac{1}{(d-1)!}\frac{1}{(c_1t)^{p-1}+(c_2t)^{q-1}}\left[ -\log\left((c_1t)^{p-1} + (c_2t)^{q-1}\right) \right]^{d-1}.
		\end{equation*}
		From Fubini's theorem we have:
		\begin{equation*}
			\begin{split}
				\mathcal{I}(c_1) &=p(p-1)q(q-1)c_1^{p-2}\gamma(c_1)^{q-2} \int_{[0,T]^2} h(t;c_1,\phi(c_1)) h(s;c_1,\gamma(c_1)) t^{2(p-1)} s^{2(q-1)} \,dt\,ds \\ &-p(q-1)q(p-1)c_1^{p-2}\phi(c_1)^{q-2} \int_{[0,T]^2} h(t;c_1,\phi(c_1)) h(s;c_1,\gamma(c_1)) t^{p+q-2} s^{p+q-2} \,dt\,ds.
			\end{split}
		\end{equation*}
		When level sets intersect we have $\phi(c_1)=\gamma(c_1)$. In this situation we can factorize the terms outside the integral and notice that the sign of $\mathcal{I}$ depends only on the sign of
		\begin{equation*}
			\begin{split}
				&\int_{[0,T]^2} h(t;c_1,\phi(c_1)) h(s;c_1,\gamma(c_1))\left( t^{2(p-1)} s^{2(q-1)} - t^{p+q-2} s^{p+q-2} \right)\,dt\,ds \\
				=&\int_{[0,T]^2} h(t;c_1,\phi(c_1)) h(s;c_1,\gamma(c_1)) t^{p-2}s^{q-2}\left( t^p s^q - t^q s^p \right)\,dt\,ds.\\
			\end{split}
		\end{equation*}
		In order to simplify the notation once again, we set $H(t,s;c_1) = h(t;c_1,\phi(c_1)) h(s;c_1,\gamma(c_1))$. Letting $T_1 = [0,T]^2 \cap \{t>s\}$ and $T_2 = [0,T]^2 \cap \{t<s\}$, we can split the above integral in two parts:
		\begin{equation*}
			\begin{split}
				\int_{T_1} H(t,s;c_1)t^{p-2}s^{q-2}\left( t^p s^q - t^q s^p \right)\,dt\,ds + \int_{T_2} H(t,s;c_1)t^{p-2}s^{q-2}\left( t^p s^q - t^q s^p \right)\,dt\,ds.
			\end{split}
		\end{equation*}
		Then, considering the change of variables that swaps $t$ and $s$ in the second integral, the domain of integration becomes $T_1$ and since $H$ is symmetric in $t$ and $s$, we have that the previous quantity is equal to
		\begin{equation*}
			\begin{split}
				&\int_{T_1} H(t,s;c_1)\left( t^{p-2} s^{q-2} - t^{q-2} s^{p-2} \right) \left( t^p s^q - t^q s^p \right)\,dt\,ds\\
				&= \int_{T_1} H(t,s;c_1)\dfrac{1}{t^2 s^2}\left( t^p s^q - t^q s^p \right)^2 \,dt\,ds,
			\end{split}
		\end{equation*}
		which is strictly positive, as desired.
		
		We are now in the position to prove the uniqueness of multipliers. First of all, since $(\phi-\gamma)'<0$ whenever $\phi(c_1) = \gamma(c_1)$, for every intersection point there exists $\delta>0$ such that $\phi(t) > \gamma(t)$ for $t \in (c_1 - \delta, c_1)$ and $\phi(t) < \gamma(t)$ for $t \in (c_1, c_1 + \delta)$.\\
		Define $c_1^* \coloneqq \sup \{c_1 \in [0,c_{1,f}] : \forall t \in [0,c_1] \ \phi(t) \geq \gamma(t)\}$. This is an intersection point between $\phi$ and $\gamma$ (if $\phi(c_1^*) > \gamma(c_1^*)$ due to continuity there would be $\varepsilon > 0$ such that $\phi(c_1^*+\varepsilon) > \gamma(c_1^*+\varepsilon)$ which contradicts the definition of $c_1^*$) and it is the first one, because we saw that after every intersection point there is an interval where $\phi < \gamma$. Lastly, since $\phi(0) > \gamma(0)$ and $\phi(c_{1,f}) = 0 < \gamma(c_1,f)$, we have that $0 < c_1^* < c_{1,f}$.\\
		Suppose now that there is a second point of intersection $\tilde{c}_1$ after the first one. Since immediately after $c_1^*$ we have that $\phi$ becomes smaller than $\gamma$, this second point of intersection is given by $\tilde{c}_1 = \sup \{c_1 \in [c_1^*,c_{1,f}] : \forall t \in [c_1^*,c_1] \ \phi(t) \leq \gamma(t)\}$. Considering that this is an intersection point, there exists an interval before $\tilde{c}_1$ where $\phi$ is strictly greater than $\gamma$ which is absurd, hence $c_1^*$ is the only intersection point between $\phi$ and $\gamma$.\\
		Therefore, the pair $(c_1^*, \phi(c_1^*) = c_2^*)$ is the unique pair of multipliers for which 
		\begin{equation*}
			p\int_0^T t^{p-1} u(t;c_1^*,c_2^*)\,dt = A^p, \quad q\int_0^T t^{q-1} u(t;c_1^*,c_2^*)\,dt = B^q
		\end{equation*}
		and, in the end, $u(t;c_1^*,c_2^*)$ is the unique maximizer for \eqref{nonstandard variational problem formulation}.
	\end{proof}
	We are now in the position to prove the second part of Theorem \eqref{th:main theorem}.
	\begin{proof}[Proof of Theorem \ref{th:main theorem}\ref{th:main theorem ii}]
		We recall that from Theorem \ref{th: norm limitation} we have
		\begin{equation*}
			\|L_{F, \varphi} \| \leq \int_{0}^{+\infty} G(\mu(t))\,dt,
		\end{equation*}
		where $\mu(t) = |\{|F|>t\}|$ is the distribution function of $F$ and equality is achieved if and only if $F(z) = e^{i \theta} \rho(|z-z_0|)$ for some $\theta \in \R$, some $z_0 \in \R^{2d}$ and some non-increasing function $\rho : [0, +\infty) \rightarrow [0, +\infty)$. Then, Theorem \ref{nonstandard variational problem solution theorem} gives us a bound on the right-hand side, namely that:
		\begin{equation*}
			\int_0^{+\infty} G(\mu(t))\,dt \leq \int_0^{+\infty} G(u(t)) \,dt,
		\end{equation*}
		where $u$ is given by \eqref{nonstandard variational problem solution formula}. This is sufficient to prove \eqref{eq:estimate Riccardi second case dimension d}. Then, from $u$ we can reconstruct $\rho$. If $F(z) = e^{i \theta} \rho (|z-z_0|)$, then its super-level sets are balls with centre $z_0$. Since $u$ is the distribution function of $F$, for $t \in (0,T)$ and some radius $r>0$ we have:
		\begin{equation*}
			\dfrac{\pi^d r^{2d}}{d!} = u(t) = \dfrac{1}{d!}\left[ - \log(\lambda_1 t^{p-1} + \lambda_2 t^{q-1})\right]^d,
		\end{equation*}
		where the left-hand side is the measure of a 2$d$-dimensional ball of radius $r$. If we simplify this expression we obtain:
		\begin{equation*}
			\pi r^2 = -\log(\lambda_1 t^{p-1} + \lambda_2 t^{q-1}) \implies t = \psi(\pi r^2).
		\end{equation*}
		However, it is clear that $\rho(r) = t$, therefore $\rho(r) = \psi(\pi r^2)$ and, in the end, $F(z) = e^{i\theta} \psi(\pi |z-z_0|^2)$.
		
		Lastly, from Theorem \ref{th: norm limitation} we have that equality in $|\langle L_{F,\varphi} f, g \rangle| = \|L_{F,\varphi}\|$ is achieved for some normalized $f$ and $g$ if they both are of the kind \eqref{eq: optimal function Nicola-Tilli Faber-Krahn for STFT dimension d}, possibly with different $c$'s but with the same centre $(x_0, \omega_0) = z_0 \in \R^{2d}$.
	\end{proof}

	\section{Some corollaries}\label{sec:Some corollaries}

	In this section, we present some corollaries of Theorem \ref{th:main theorem}. The first one just consists in rephrasing Theorem \ref{th:main theorem} in the form of an optimal estimate for $\|L_{F,\varphi}\|$.
	\begin{cor}
		Let $p, q \in (1,+\infty)$ and let $F \in L^p(\R^{2d}) \cap L^q(\R^{2d})$. Then:
		\begin{enumerate}[label*=(\roman{*})]
			\item If $\|F\|_q / \|F\|_p \leq \kappa_q^{d(\frac{1}{q}-\frac{1}{p})} \left(\dfrac{p}{q}\right)^{\frac{d}{p}}$ then
			\begin{equation*}
				\|L_{F,\varphi}\| \leq \kappa_q^{d \kappa_q} \|F\|_q
			\end{equation*}
			with equality if and only if
			\begin{equation*}
				F(z) = A e^{i \theta} e^{-\frac{\pi}{q-1}|z-z_0|^2},
			\end{equation*}
			for some $A > 0$, $\theta \in \R$ and $z_0 \in \R^{2d}$.
			\item If $\kappa_q^{d(\frac{1}{q}-\frac{1}{p})} \left(\dfrac{p}{q}\right)^{\frac{d}{p}} < \|F\|_q / \|F\|_p < \kappa_p^{d(\frac{1}{q}-\frac{1}{p})} \left(\dfrac{p}{q}\right)^{\frac{d}{q}}$
			then
			\begin{equation*}
				\|L_{F,\varphi}\| \leq \int_0^{+\infty} G(u(t)) \,dt,
			\end{equation*}
			where $u(t) = \frac{1}{d!} \left[ \Log (\lambda_1 t^{p-1} + \lambda_2 t^{q-1})\right]^d$ and $\lambda_1, \lambda_2 > 0$ are uniquely determined by
			\begin{equation*}
				p \int_0^{+\infty} t^{p-1} u(t) \,dt = \|F\|_p^p, \quad q \int_0^{+\infty} t^{q-1} u(t) \,dt = \|F\|_q^q.
			\end{equation*}
			Moreover, letting $T > 0$ the unique value such that $\lambda_1 T^{p-1} + \lambda_2 T^{q-1} = 1$, the function $t \mapsto -\log(\lambda_1 t^{p-1} + \lambda_2 t^{q-1})$ defined on $(0,T]$ is invertible and we denote by $\psi : [0, + \infty) \rightarrow (0,T]$ its inverse. Then, equality is achieved if and only if $F(z) = e^{i \theta} \psi(\pi |z-z_0|^2)$ for some $\theta \in \R$ and $z_0 \in \R^{2d}$.
			
			\item If $\|F\|_q / \|F\|_p \geq \kappa_p^{d(\frac{1}{q}-\frac{1}{p})} \left(\dfrac{p}{q}\right)^{\frac{d}{q}}$ then
			\begin{equation*}
				\|L_{F,\varphi}\| \leq \kappa_p^{d \kappa_p} \|F\|_p
			\end{equation*}
			with equality if and only if
			\begin{equation*}
				F(z) = A e^{i \theta} e^{-\frac{\pi}{p-1}|z-z_0|^2},
			\end{equation*}
			for some $A > 0$, $\theta \in \R$ and $z_0 \in \R^{2d}$.
		\end{enumerate}
	\end{cor}
	Before presenting a further consequence of Theorem \ref{th:main theorem}, we recall that $\V_{\varphi} : L^2(\R^d) \rightarrow L^2(\R^{2d})$ is an isometry since $\| \varphi \|_2 = 1$ (see \cite{grochenig}). Therefore, $|\V_{\varphi} f(x,\omega)|^2$ can be regarded as the time-frequency energy density of $f$ and is thus important to have some estimates for this quantity, both from a mathematical and physical point of view. We also recall (see \cite{bergh}) that, if $p,q \in (1,+\infty)$, the space $L^p(\R^d) + L^q(\R^d)$ is defined as
	\begin{equation*}
		L^p(\R^d) + L^q(\R^d) = \{f+g \colon f \in L^p(\R^d),\, g \in L^q(\R^d)\}.
	\end{equation*}
	This is a Banach space with the norm
	\begin{equation*}
		\|f\|_{L^p + L^q} = \inf \{\|f_1\|_p + \|f_2\|_q \colon f_1 \in L^p(\R^d),\, f_2 \in L^q(\R^d),\, f = f_1 + f_2\},
	\end{equation*}
	and its dual space can be identified with $L^{p'}(\R^d) \cap L^{q'}(\R^d)$, where $p'$ and $q'$ are the conjugate exponents of $p$ and $q$, respectively, with the following norm
	\begin{equation*}
		\|f\|_{L^{p'} \cap L^{q'}} = \max \{\|f\|_{p'},\, \|f\|_{q'}\}.
	\end{equation*}
	Then, thanks to Theorem \ref{th:main theorem} we are able to give an optimal estimate for $\| |\V_{\varphi} f |^2 \|_{L^p + L^q}$, which reduces to the $d$-dimensional version of \eqref{eq:Lieb's inequality} when $p=q$.
	\begin{cor}\label{cor:estimate for spectrogram}
		Let $p,q \in (1,+\infty)$. Then, for every $f \in L^2(\R^d)$ it holds that
		\begin{enumerate}[label*=(\roman{*})]
			\item If $p^{d(\kappa_p-\kappa_q)} \left(\dfrac{\kappa_q}{\kappa_p}\right)^{d \kappa_q} \leq 1$ then
			\begin{equation*}
				\| |\V_{\varphi} f |^2 \|_{L^p + L^q} \leq \kappa_p^{d\kappa_p} \|f\|_2^2.
			\end{equation*}
			\item If $q^{d(\kappa_p-\kappa_q)}\left(\dfrac{\kappa_q}{\kappa_p}\right)^{d\kappa_p} \geq 1$ then
			\begin{equation*}
				\| |\V_{\varphi} f |^2\|_{L^p + L^q} \leq \kappa_q^{d\kappa_q} \|f\|_2^2.
			\end{equation*}
			\item If $q^{d(\kappa_p-\kappa_q)}\left(\dfrac{\kappa_q}{\kappa_p}\right)^{d\kappa_p} < 1 < p^{d(\kappa_p-\kappa_q)} \left(\dfrac{\kappa_q}{\kappa_p}\right)^{d \kappa_q}$ then
			\begin{equation*}
				\| |\V_{\varphi} f|^2 \|_{L^p + L^q} \leq \left( \int_{0}^{+\infty} G(u(t)) \, dt \right) \|f\|_2^2,
			\end{equation*}
			where $u(t)$ is given by \eqref{eq:optimal distribution function} and $\lambda_1, \lambda_2$ are uniquely determined by \eqref{eq:equations for multipliers} with $A=B=1$.
		\end{enumerate}
		Finally, in every case, equality is achieved if and only if $f$ is of the kind \eqref{eq: optimal function Nicola-Tilli Faber-Krahn for STFT dimension d} for some $c \in \C$ and some $z_0 = (x_0, \omega_0) \in \R^{2d}$.
	\end{cor}
	\begin{proof}
		Given $f \in L^2(\R^d)$ we have:
		\begin{align*}
			\| |\V_{\varphi} f |^2 \|_{L^p + L^q} &= \max_{ \|F\|_{L^{p'} \cap L^{q'}} \leq 1} \lvert \langle F, 	|\V_{\varphi} f |^2 \rangle \rvert = \max_{ \|F\|_{L^{p'} \cap L^{q'}} \leq 1} \lvert \langle L_{F,\varphi}f, f \rangle \rvert \\
												  &\leq \left(\sup_{ \|F\|_{L^{p'} \cap L^{q'}} \leq 1} \|L_{F,\varphi}\|\right) \|f\|_2^2,
		\end{align*}
		and the result directly follow from Theorem \ref{th:main theorem} with exponents $p'$ and $q'$ and $A=B=1$.
	\end{proof}
	
	Finally, we recall that Fock space $\Fock^2(\C^d)$ is the space of all entire functions $F$ on $\C^d$ for which the norm
	\begin{equation*}
		\| F \|_{\Fock}^2 = \int_{\C^d} |F(z)|^2 e^{- \pi |z|^2} \, dz
	\end{equation*}
	is finite and that the Bargmann transform $\Barg \colon L^2(\R^d) \rightarrow \Fock^2(\C^d)$ defined as
	\begin{equation*}
		\Barg  f(z) = 2^{\frac{d}{4}} \int_{\R^d} f(t) e^{2 \pi t \cdot z - \pi |t|^2 - \frac{\pi}{2}z^2} \, dt
	\end{equation*}
	is a unitary operator (see \cite{grochenig}). The Bargmann transform is closed related with the STFT with Gaussian window thanks to the following relation:
	\begin{equation}\label{eq:formula STFT and Bargmann transform}
		\V_{\varphi} f (x, -\omega) = e^{\pi i x \cdot \omega} \Barg f(z) e^{-\pi |z|^2/2}, \quad z = x + i \omega \in \C^d.
	\end{equation}
	\begin{cor}\label{cor:estimate for functions in Fock space}
		Let $p, q \in (1,+\infty)$. Then, for every $F \in \Fock^2(\C^d)$ it holds that:
		\begin{enumerate}[label*=(\roman{*})]
			\item If $p^{d(\kappa_p-\kappa_q)} \left(\dfrac{\kappa_q}{\kappa_p}\right)^{d \kappa_q} \leq 1$ then
			\begin{equation*}
				\| |F|^2 e^{- \pi | \cdot |^2 } \|_{L^p + L^q} \leq \kappa_p^{d\kappa_p} \|F\|_{\Fock}^2.
			\end{equation*}
			\item If $q^{d(\kappa_p-\kappa_q)}\left(\dfrac{\kappa_q}{\kappa_p}\right)^{d\kappa_p} \geq 1$ then
			\begin{equation*}
				\| |F|^2 e^{- \pi | \cdot |^2} \|_{L^p + L^q} \leq \kappa_q^{d\kappa_q} \|F\|_{\Fock}^2.
			\end{equation*}
			\item If $q^{d(\kappa_p-\kappa_q)}\left(\dfrac{\kappa_q}{\kappa_p}\right)^{d\kappa_p} < 1 < p^{d(\kappa_p-\kappa_q)} \left(\dfrac{\kappa_q}{\kappa_p}\right)^{d \kappa_q}$ then
			\begin{equation*}
				\| |F|^2 e^{- \pi | \cdot |^2} \|_{L^p + L^q} \leq \left( \int_{0}^{+\infty} G(u(t)) \, dt \right) \|F\|_{\Fock}^2,
			\end{equation*}
			where $u(t)$ is given by \eqref{eq:optimal distribution function} and $\lambda_1, \lambda_2$ are uniquely determined by \eqref{eq:equations for multipliers} with $A=B=1$.
		\end{enumerate}
	\end{cor}
	\begin{proof}
		Since the Bargmann transform is a unitary operator, there exists $f \in L^2(\R^d)$ such that $\Barg f = F$ and $\|f\|_2 = \| F \|_{\Fock}$. Therefore, thanks to \eqref{eq:formula STFT and Bargmann transform} it holds that
		\begin{equation*}
			| \V_{\varphi} f (x, -\omega)|^2 = |F(z)|^2 e^{- \pi |z|^2}
		\end{equation*}
		and the proof is complete applying Corollary \ref{cor:estimate for spectrogram}.
	\end{proof}
	\appendix
	\begin{center}
		\begin{figure}[h]
			\includegraphics[scale=0.4]{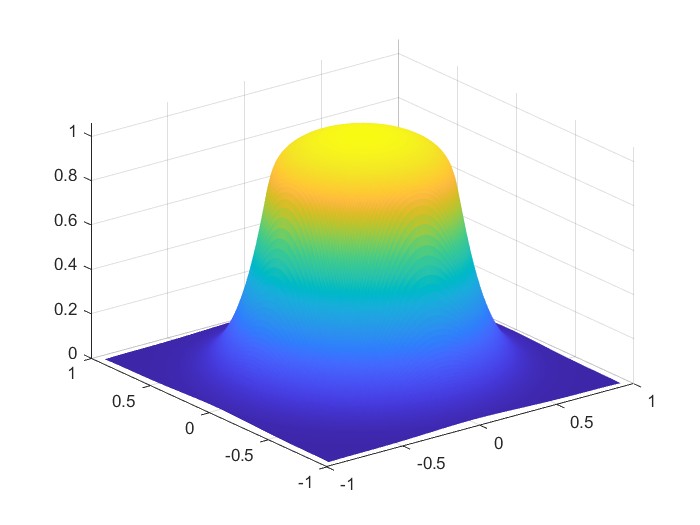}
			\caption{An optimal weight function with $d=1$, $p=1.5$, $q=20$, $A=B=1$. The plot was created with MATLAB R2022b.}
			\label{fig:optimal weight function}
		\end{figure}
	\end{center}
	
\nocite{*}
\printbibliography

\end{document}